\documentclass[12pt]{amsart}
\usepackage[doublespacing]{setspace}
\usepackage{fullpage}
\usepackage{url}
\usepackage{hyperref}

\newcommand{\textd}{\text{d}}
\newcommand{\bbr}{\mathbb{R}}

\newcommand{\bfone}{\mathbf{1}}

\newcommand{\bbn}{\mathbb{N}}

\newcommand{\ep}{\epsilon}

\newcommand{\bbp}{\mathbb{P}}

\newcommand{\support}{\text{support}}
\DeclareMathOperator*{\argmax}{arg\,max}

\swapnumbers
\theoremstyle{plain}
\newtheorem{theorem}{Theorem}[section]
\newtheorem{proposition}[theorem]{Proposition}
\newtheorem{lemma}[theorem]{Lemma}

\newtheorem{corollary}[theorem]{Corollary}

\theoremstyle{definition}
\newtheorem{definition}[theorem]{Definition}
\newtheorem{example}[theorem]{Example}

\theoremstyle{remark}
\newtheorem{remark}[theorem]{Remark}

\newtheorem*{claim*}{Claim}

\newcommand{\barx}{\bar x}
\newcommand{\bary}{\bar y}
\begin{document}

\title[Lipschitz games]{Lipschitz games}
\author{Yaron Azrieli}\email{azrieli.2@osu.edu}\address{Department of Economics\\ The Ohio State University}
\author{Eran Shmaya}\email{e-shmaya@kellogg.northwestern.edu}\address{Kellogg School of Management\\ Northwestern University}
\thanks{Thanks
to Joyee Deb and Ehud Kalai for helpful discussions about large games, to Ezra Einy for informing us about the Shapley-Folkman Theorem, to Roger Myerson for Remark~\ref{meyerson}, to Ron Peretz for comments on a previous version of the manuscript, to Marcin Peski for suggesting the network games application, to David Schmeidler for suggesting the current proof of Theorem~\ref{thm-anonymous} (which is simpler than our original proof) and to Rakesh Vohra for informing us about Gale-Berlekamp Switching Game.} 
\begin{abstract}
The \emph{Lipschitz constant} of a finite normal--form game is the maximal change in some player's payoff when a single opponent changes his strategy. We prove that games with small Lipschitz constant admit pure $\ep$-equilibria, and pinpoint the maximal Lipschitz constant that is sufficient to imply existence of pure $\ep$-equilibrium as a function of the number of players in the game and the number of strategies of each player. Our proofs
use the probabilistic method.
\end{abstract}

\maketitle


\section{introduction}\label{sec-intro}
The use of mixed strategies in game theory is often criticized, since one is reluctant to view rational decisions as based on coin tossing (see, e.g., Rubinstein's discussion \cite[Section 3]{Rub-91}). Thus the long tradition in the literature of identifying classes of games that admit pure equilibrium \cite{Mil-Rob-90, Mon-Sha-96, Ros-73}. Pure Nash equilibria do not suffer from the conceptual difficulty of mixed strategies, which make this solution concept more appealing.

This paper identifies a new class of games that admit pure $\ep$-equilibrium. We call these games \emph{Lipschitz games}, since their characterizing property is that the payoff of a player does not change much when a single opponent changes his strategy; thus, each player's payoff function is Lipschitz in her opponents' strategy profile.
More precisely, we define the \emph{Lipschitz constant of a game} to be the maximal change in some player's payoff when a single opponent changes his strategy. 
If the Lipschitz constant is $o(1/n)$ when the number of players $n$ grows, then any game with sufficiently many players admits a pure $\ep$-equilibrium, since in this case a player's payoff is essentially independent of her opponents' strategy profile (Proposition \ref{sure}). We show in this paper that,
%
when the number of strategies for every player is fixed, a Lipschitz constant of $o(1/\sqrt{n\log n})$ is already sufficient to guarantee pure $\ep$-equilibrium (Theorem~\ref{yes}) and give an example for games with  Lipschitz constant of order $O(1/\sqrt{n})$ with no such equilibrium (Theorem~\ref{no}).

Our results give new perspective on the literature on large (with many players) games. Large games are used as a modeling apparatus that captures two intuitive aspects of interactions of many economic agents: Continuity and anonymity. Continuity means that an agent's behavior has only small impact on the utility of other agents. Anonymity means that each agent's utility depends only on the aggregate behavior of the other agents. A typical example is a congestion game with many players, such as drivers choosing among several roads: A driver's utility depends on the other drivers' choices only through the congestion they form on each road (anonymity), and each driver has a negligible impact on this congestion (continuity).

One of the earliest models of strategic interaction with many players, due to Schmeidler~\cite{Sch-73}, is games with a continuum of players: The set of players is given by a non--atomic measure space; continuity is reflected by the fact that opponent's strategy profiles which are equal almost everywhere yield the same payoff, and therefore a change in one opponent's strategy does not affect a player's payoff; and anonymity is reflected by the fact that opponents' strategy profile affect a player's payoff only through its integral. Schmeidler shows that pure Nash equilibrium exists under these assumptions.
Kalai~\cite{Kal-04} considers large finite games under continuity and anonymity assumptions: Roughly speaking, for the special case of complete information games with deterministic types, a player's payoff is a continuous function of her own strategy and of the empirical distribution of opponent's types and strategies. Kalai's result implies existence of pure $\ep$-equilibrium in such games with sufficiently many players. Several subsequent papers \cite{Car-08,Car-Pod-10-1, Deb-Kal-10, Gra-Rei-10} 
relax some of Kalai's assumptions and obtain similar results.

Note that the formulation of the continuity assumption in Kalai's paper already assumes anonymity. In this paper continuity is captured by the Lipschitz constant of the game, and is therefore not tied to anonymity. Kalai's argument is based on \emph{self-purification}: the realized profile of a mixed Nash equilibrium is, with high probability, a pure $\ep$-equilibrium, and it has been a folk knowledge that the argument is based on continuity alone. This paper pinpoints the limit behavior of the Lipschitz constant required for the self-purification argument to hold and for pure equilibrium to exist.


A common situation where the continuity assumption holds without the anonimity is network games~\cite[Sections 9.2-9.3]{Jac-08}. In these games players are identified with nodes of a graph, and the payoff of each player depends in a Lipschitz continuous way on the empirical distribution of strategies of her neighbors. These games are not anonymous, but it follows from our result that if the set of neighbors of each player is sufficiently large relative to the total number of players, then a pure approximate equilibrium exists.

Section~\ref{sec-setup} contains the necessary definitions and the main results. The proofs of the existence result and of its tightness are in Sections~\ref{sec-proof-yes} and~\ref{sec-proof-no}. In Section \ref{sec-strategy} we show that if the number of strategies of each player is unbounded, then it is not possible to improve the trivial bound of $o(1/n)$ to guarantee pure $\ep$-equilibrium existence. 
In Section \ref{sec-anonymous} we consider the special case of anonymous games and show that for such games the required Lipschitz constant is $o(1)$, that is independent of the number of players in the game, even though the self-purification argument fails. Appendix \ref{sec-converse} shows that existence of pure approximate equilibria in Lipschitz games implies Nash's theorem on existence of mixed equilibrium.
\section{Main results}\label{sec-setup}
An $n$--player game in normal form $G$ is given by finite sets $\{A_i\}_{i=1}^n$ of \emph{strategies} and by \emph{payoff functions} $\left\{f_i:A\rightarrow \mathbb{R} \right\}_{i=1}^n$, where $A=\prod_{i=1}^n A_i$ is the set of \emph{strategy profiles}. A \emph{mixed strategy} for player $i$ is a probability distribution over $A_i$. 
Nash equilibrium and $\ep$--Nash equilibrium (in pure or mixed strategies) are defined as usual.

For each $i$, we view the product space $A_{-i}=\prod_{j\neq i}A_j$ as a metric space, with the metric
$\rho(a_{-i}',a_{-i}'')=\#\{1\le j \le n ~\colon~ j \neq i,~ a'_j\neq a''_j\}$.
\begin{definition}\label{def-impact}
The \emph{Lipschitz constant} of $G$ is given by $$\delta(G)=\max\{|f_i(a_i,a_{-i}')-f_i(a_i,a_{-i}'')|\},$$
where the maximum ranges over all players $i$, all strategies $a_i\in A_i$ and all pairs $a_{-i}',a_{-i}''$ of opponents' strategy profiles such that $\rho(a_{-i}',a_{-i}'')=1$.
\end{definition}
Games with Lipschitz constant $\delta$ have the property that a player's payoff does not change by more than $\delta$ when one opponent changes his strategy. This implies that for each $i$ and each $a_i\in A_i$ the function $f_i(a_i, \cdot)$ is $\delta$--Lipschitz on $A_{-i}$. We denote by $L(n,m,\delta)$ the set of games with $n$ players, at most $m$ strategies for every player and Lipschitz constant of at most $\delta$.

\begin{remark}
A notion of `influence of a player' appears also in \cite{Aln-Smo-00} for a different purpose than in this paper. In our setup, the Lipschitz constant is a property of the game $G$, while in their setup the influence is a property of the distribution of players' types and of the mechanism.
\end{remark}

The following Proposition~\ref{sure} establishes a trivial sufficent condition for existence of pure $\ep$-equilibrium when the Lipschitz constant is so small that a player's payoff essentially depends only on her own strategy.
\begin{proposition}\label{sure}
Let $\ep>0$ and $m,n\in \bbn$. Then every game in $L(n,m,\delta)$ for $\delta = \ep/2n$ admits a pure $\ep$--equilibrium.
\end{proposition}
\begin{proof}
Let $a\in A$ be an arbitrary strategy profile and let $a^\ast$ be a strategy profile such that $a^\ast_i$ is a best response to $a_{-i}$ for every player $i$. Then $a^\ast$ is an $\ep$--equilibrium. Indeed, for every player $i$ and every deviation $d\in A_i$ one has
\begin{multline*}
f_i(d,a^\ast_{-i})\leq f_i(d,a_{-i})+(n-1)\ep/2n\leq\\ f_i(a^\ast_i,a_{-i})+(n-1)\ep/2n\leq f_i(a^\ast_i,a^\ast_{-i})+2(n-1)\ep/2n<f_i(a^\ast_i,a^\ast_{-i})+\ep,
\end{multline*}
where the first and third inequalities follow from the Lipschitz property of $f_i$ and the second inequality follows from the definition of $a^\ast$.
\end{proof}
\begin{corollary}Let $\delta:\bbn\rightarrow [0,1]$ be such that $\delta(n)=o(1/n)$. Then for every $\ep>0$, there is $N$ such that every game with $n\ge N$ players and Lipschitz constant smaller than $\delta(n)$ admits a pure $\ep$-equilibrium.\end{corollary}

Since our main interest is in games with many players, we would like to think of the size of the strategy set as fixed and increase the number of players to infinity. Theorem~\ref{yes} establishes existence of pure approximate equilibrium in games with sufficiently small Lipschitz constant, which is asymptotically much larger than the Lipschitz constant in Proposition~\ref{sure}. Theorem~\ref{no} shows the almost tightness of Theorem~\ref{yes}.
\begin{theorem}\label{yes}
Let $\ep>0$ and $m,n\in \mathbb{N}$. Then every game in $L(n,m,\delta)$ for
$\delta = \ep/\sqrt{8n\log (2mn)}$ admits a pure $\ep$--equilibrium.
\end{theorem}
\begin{corollary}\label{yes-corollary}
Fix $m$ and let $\delta:\bbn\rightarrow [0,1]$ be such that $\delta(n)=o(1/\sqrt{n\log n})$. Then for every $\ep>0$, there is $N$ such that every game with $n\ge N$ players, $m$ strategies for each player and Lipschitz constant smaller than $\delta(n)$ admits a pure $\ep$-equilibrium.
\end{corollary}
\begin{remark}
Deb and Kalai's \cite{Deb-Kal-10} result implies that any game with Lipschitz constant of order $O(1/n)$ admits a pure $\ep$-equilibrium.\end{remark}
\begin{theorem}\label{no}
For every even $n$ large enough there is a game in $L(n,2,60/\sqrt n)$ with no pure 1/3--equilibrium. Moreover, the payoffs can be chosen to be bounded in $[-1,1]$.
\end{theorem}
\begin{remark}\label{meyerson}One can define the Lipschitz constant of a game in a slightly different way, to get a stronger result: For every game $G$ let
\[\eta(G)=\max\{|\left(f_i(a_i',a_{-i}')-f_i(a_i'',a_{-i}')\right)-\left(f_i(a_i',a_{-i}'')-f_i(a_i'',a_{-i}'')\right)|\}\]
where the maximum ranges over all players $i$, all strategies $a_i',a_i''\in A_i$ and all pairs $a_{-i}',a_{-i}''$ of opponents' strategy profiles such that $\rho(a_{-i}',a_{-i}'')=1$. Then $\eta(G)\leq 2\delta(G)$ so that if $\delta(G)$ is small then so is $\eta(G)$, but $\eta(G)$ may be small even for games with large Lipschitz constant. Still, Theorem~\ref{yes} has the following simple corollary:\end{remark}
\begin{corollary}
Let $\ep>0$ and $m,n\in \mathbb{N}$. Then every game $G$ with $n$ players, at most $m$ strategies per player and with $\eta(G)\le \ep/\sqrt{8n\log (2mn)}$ admits a pure $\ep$--equilibrium.\end{corollary}
\begin{proof}Let $G$ be a game with $\eta(G)\le \ep/\sqrt{8n\log (2mn)}$. For every player $i$ let $\bar a_i\in A_i$ and let $g_i:A\rightarrow \bbr$ be the payoff functions given by $g_i(a)=f_i(a)-f_i(\bar a_i,a_{-i})$. Let $G'$ be the game with action sets $A_i$ and payoff functions $g_i$. Then $\delta(G')\le \eta(G)$ so that $G'$ has a pure $\ep$-equilibrium by Theorem~\ref{yes}, but every such an equilibrium is also an $\ep$-equilibrium in $G$.\end{proof}

\section{Proof of Theorem~\ref{yes}}\label{sec-proof-yes}
The proof follows Kalai's \cite{Kal-04} proof for anonymous games. The gist of the proof is the observation that a random realized strategy profile of a mixed Nash equilibrium is, with high probability, an $\ep$-equilibrium. Kalai dubs this phenomena \emph{self--purification}. In Kalai's anonymous setup, the argument relies on the law of large numbers. Here we use the related fact that the value of a Lipschitz function under a random input is concentrated around its expectation. A similar argument appears in Deb and Kalai~\cite{Deb-Kal-10} and Carmona and Podczeck~\cite{Car-Pod-10-1}.
\begin{proposition}\cite[Corollary 1.17]{Led-01}\label{liphshitz-prop}
Let $A_1,\dots,A_n$ be finite sets and let $\mu=\mu_1\times\dots\times\mu_n$ be a product probability measure over $A = \prod_i A_i$. Let $F:A\rightarrow\bbr$ be a real valued function such that $|F(a)-F(a')|\leq 1$ for every $a,a'\in A$ such that $\rho(a,a')=1$. Then for every $r>0$
\[\mu \left(\left\{a\colon F(a)\geq \int F\textd \mu + r\right\}\right) \leq e^{-r^2/2n}.\]
\end{proposition}
\begin{remark}\label{rem-lipshitz}\begin{enumerate}
\item If the Lipschitz constant for $F$ is $\delta$ (instead of 1 as in the above formulation) then by considering the function $F/\delta$ the bound on the probability becomes $e^{-r^2/2n\delta^2}$.\\
\item By applying the same bound to $-F$ one gets
\[\mu \left(\left\{a\colon \left| F(a)-\int F\textd \mu \right| \geq r\right\}\right) \leq 2e^{-r^2/2n}.\]\end{enumerate}
\end{remark}
Consider a game in $L(n,m,\delta)$ with $\delta = \ep/\sqrt{8n\log(2mn)}$. Let $(\mu_1,\dots,\mu_n)$ be a mixed strategy Nash equilibrium of the game. Thus, each $\mu_i$ is a probability distribution over $A_i$ and
\begin{equation}\label{nash}
\support(\mu_i)\subseteq\argmax_{a_i\in A_i}\int f_i(a_i,\tau)\mu_{-i}(\textd \tau),\end{equation}
where $\mu_{-i}=\prod_{j\neq i}\mu_j$.

For every player $i$ and every strategy $h\in A_i$ let $E_{i,h}\subseteq A$ be the set of all strategy profiles $a$ such that, if player $i$ plays $h$ against $a_{-i}$ her payoff is roughly the same as her expected payoff when she plays $h$ and the opponents play their Nash equilibrium strategy:
\[E_{i,h}= A_i \times \left\{a_{-i}\in A_{-i} \colon \left\lvert f_i(h,a_{-i})-\int f_i(h,\tau)\mu_{-i}(\textd \tau)\right\rvert\leq \ep/2 \bigr.\right\}.\]
From Proposition~\ref{liphshitz-prop}, the Lipschitz property of $f_i(h,\cdot)$ and the choice of $\delta$ it follows that
\[\mu(E_{i,h}^c)\leq 2\exp(-\ep^2/8(n-1)\delta^2)<1/nm\] for every player $i$ and every $h\in A_i$. Since there are at most $mn$ such events $E_{i,h}$, it follows that $\mu(\cap E_{i,h})>0$. Let $a^\ast$ be a strategy profile such that $a^\ast\in\support(\mu)$ and $a^\ast\in \cap E_{i,h}$. We claim that $a^\ast$ is an $\ep$--equilibrium. Indeed, for every player $i$ and every deviation $d\in A_i$ one has
\[f_i(d,a^\ast_{-i})\leq \int f_i(d,\tau)\mu_{-i}(\textd \tau)+\ep/2\leq \int f_i(a^\ast_i,\tau)\mu_{-i}(\textd \tau)+\ep/2\leq f_i(a^\ast_i,a^\ast_{-i})+\ep\]
where the first inequality follows from the fact that $a^\ast\in E_{i,d}$, the second from~(\ref{nash}) and the third from the fact that $a^\ast\in E_{i,a^\ast_i}$.

%
\section{Proof of Theorem \ref{no}}\label{sec-proof-no}
Let the number of players be $n=2k$ and let the set of strategies for each player be $\{+1,-1\}$. We divide the players into two groups of $k$ players, females and males, and denote their strategy profiles by $\bar x=(x_1,\dots,x_k)$ and $\bar y = (y_1,\dots,y_k)$ respectively, viewed as row vectors. Fix some constant $\delta>0$. We consider games that can be described by a $k\times k$ matrix $M=\{m_{ij}\}$ with entries $\pm 1$.
The payoff for female $i$ is $K(u_i)$ where $K(t)=t$ for $|t|\leq 1$ and $K(t)=t/|t|$ for $|t|>1$, and $u_i$, the \emph{untruncated payoff of female $i$}, is given by
\[u_i (\bar x,\bar y) = \delta x_i\sum_j m_{ij}y_j = \delta x_i\cdot (M\bar y^T)_i.\]
The payoff for male $j$ is given by $K(v_j)$ where the untruncated payoff of male $j$ is given by
\[v_j (\bar x,\bar y) = -\delta y_j\sum_i m_{ij}x_i = -\delta y_j\cdot (\bar x M)_j.\]

Thus, up to truncation (which ensures that payoffs are bounded in $[-1,1]$), every player plays matching pennies against each player of the opposite gender. Each player uses the same coin in all their matching pennies games, and the parameter $m_{ij}$ dictates which player wants to match in the game between female $i$ and male $j$. Notice that the Lipschitz constant in every such game is at most $2\delta$. 
Also, from the definitions of the untruncated payoff it follows that
\begin{equation}\label{zerosum}
\sum_i u_i(\bar x,\bar y)=-\sum_j v_j(\bar x,\bar y)
\end{equation}
for every profile $(\barx,\bary)$.

It is helpful to think about the game as played on a $k\times k$ array of lights, each can be either on or off, with switches for each row and each column, such that if a switch of a row or column is pulled then all the lights in that row or column are switched (from on to off or from off to on). Every female player controls a row switch and every male player controls a column switch. The purpose of a female players is to have as many lights `on' in her row, and the purpose of a male player is to have as many lights `off' in his column. In a similar framework, Gale and Berlekamp introduced a two-player zero-sum game with perfect information, where Player 1 chooses the initial configuration of the lights, and then Player 2 chooses which switches to pull. The goal of player 2 is to minimize the number of lights.

The following lemma says that there exists some initial configuration of the lights such that whatever switches are performed on the rows, many columns will be unbalanced (i.e. will have much more lights on than off or vice versa).
\begin{lemma}\label{thelemma}
For sufficiently large $k$ there exists $k\times k$ matrix $M$ with entries in $\{+1,-1\}$ such that
\begin{equation}\label{such-that-lemma}
\# \left\{ 1\le j \le k ~:~  | (\bar{x}M)_j | > \frac{\sqrt{k}}{20} \right\}  > k/3
\end{equation}
for every row vector $\bar x$ of length $k$ with entries in $\{+1,-1\}$.
\end{lemma}
\begin{proof}[Proof of Theorem~\ref{no}]By Lemma~\ref{thelemma}
for every sufficiently large $k$ and $\delta = 20/\sqrt{k} < 30/\sqrt{n}$ there exits a $k\times k$ matrix $M$ with the property
\begin{equation}\label{such-that-thm}
\#\left\{ 1\le j \le k ~:~  | (\bar{x}M)_j | > \frac{1}{\delta} \right\}  > k/3
\end{equation}
for every strategy profile $\bar x$ of the females.

We claim that if $M$ satisfies~(\ref{such-that-thm}) then the game admits no $1/3$-equilibrium. Indeed, fix a strategy profile $\bar x$ for the females and let $\bar y$ be a profile such that all the males play $1/3$--best response to $\bar x$. Since by changing his strategy a player inverts the sign of his untrancated payoff, it follows that $v_j\ge -1/6$ for every male $j$. In particular, $v_j\ge 1$ whenever $|v_j|\ge 1$. 
Therefore, by (\ref{such-that-thm}) it follows that
\[\sum_j v_j(\bar x,\bar y) > k/3 * 1 + (2k/3) * (-1/6) > k/6.\]
By (\ref{zerosum}), it follows that $u_i(\bar x,\bar y) < -1/6$ for some female $i$. Since every player inverts the sign of her payoff by changing strategy it follows that female $i$ does not $1/3$ best-respond to $\bar y$.\end{proof}
The proof of Lemma~\ref{thelemma} uses the probabilistic method. Alon and Spencer \cite[Section 2.5]{Alo-Spe-08} use the probabilistic method to prove that for every initial configuration it is possible to switch lights to unbalance the matrix. We turn the probabilistic method `on its head' to prove that there exists some initial configuration for which any switching will result in an unbalanced matrix. The argument follows the proof of lower bound in the classical discrepancy problem~\cite[Section 13.4]{Alo-Spe-08}.
\begin{proof}[Proof of Lemma~\ref{thelemma}]
Fix $k$ and let $M_{k\times k} = \left\{m_{ij}\right\}$
where $m_{ij}$ are independent random signs. For a fixed
$\bar x$, the entries of $z=\bar x \cdot M$ are $i.i.d$,
and each $z_j$ is distributed like the sum of $k$
independent random signs. Thus, by 
the central limit theorem
$$\bbp\left(|z_j| \le \frac{\sqrt k}{20} \right)\xrightarrow{k\rightarrow\infty}\frac{1}{\sqrt{2\pi}}\int_{-1/20}^{1/20}e^{-\tau^2/2}\textd\tau<1/25$$

Fix a row vector $\bar x$ of length $k$ with entries in $\{+1,-1\}$ and let $E_{\bar{x}}$ be the event that (\ref{such-that-lemma}) is not satisfied. Then $E_{\bar{x}}$ is the event that there are more than $2k/3$ successes in $k$ independent trials with probability of success smaller than $1/25$. From Chernoff inequality \cite[Theorem A.1.4]{Alo-Spe-08} we get
\[\bbp(E_{\bar{x}})\leq \exp \left(-2k \left(2/3- 1/25\right)^2 \right) < 1/2^k.\]
Since there are only $2^k$ possible $\bar{x}$-s, it follows that $P\left(\cup_{\bar{x}} E_{\bar{x}} \right) < 1$. Therefore, for some choice of $M$ none of the $E_{\bar{x}}$ obtains, as desired.
\end{proof}

\section{Large strategy sets}\label{sec-strategy}
Proposition~\ref{sure} establishes a trivial existence result where the bound on the Lipschitz constant is independent of the number of strategies in the game. Proposition~\ref{no-strategy} below shows that for games with unbounded strategy sets this bound is the best possible (up to a constant).
\begin{proposition}\label{no-strategy}
For every even $n$ there is a game in $L(n, 2^{n/2}, 1/n)$ with no pure $1/8$--equilibrium.
\end{proposition}
To prove the proposition we construct another `mass version' of matching pennies. Again we divide the players into groups of females and males, but this time every player has a coin for each of her/his opponents. The strategy of every player encodes all their coins.
\begin{proof}
Let the number of players be even $n=2k$, and let the strategy set of each player be $\{+1,-1\}^k$. 
We divide the players into two groups, females and males, and denote their strategy profiles by $\bar x=(x_1,\dots,x_k)$ and $\bar y = (y_1,\dots,y_k)$ respectively. The strategy of female $i$ is given by the vector $\left(x_i[j]\right)_{j=1}^k$, and similarly the strategy of male $j$ is given by the vector $\left(y_j[i]\right)_{i=1}^k$.

The payoff to female $i$ is
\[u_i(\bar x,\bar y)=\frac{1}{4k}\sum_j x_i[j]\cdot y_j[i],\]
and the payoff to male $j$ is
\[v_j(\bar x,\bar y)=-\frac{1}{4k}\sum_i  x_i[j]\cdot y_j[i].\]
The game has Lipschitz constant $2/4k=1/n$ but no pure $1/8$--equilibrium: For every strategy profile of the opponents, every player can guarantee $1/4$. Therefore, in every $1/8$--equilibrium every player should get at least $1/8$. But this is impossible since the sum of all the players' payoffs is $0$ in every profile.
\end{proof}




\section{Anonymous games}\label{sec-anonymous}

A game $G$ is \emph{anonymous} if all the players have the same strategy set and the payoff to each player $i$ is not changed when two other players $j$ and $k$ exchange the strategies they play. In this section we show that for the class of anonymous games a much stronger result than Theorem \ref{yes} holds (assuming a fixed number of strategies). Namely, the Lipschitz constant required to guarantee existence of pure $\ep$-equilibrium is independent of the number of players in the game:
\begin{theorem}\label{thm-anonymous}
Let $\ep>0$ and $m,n\in \mathbb{N}$. Then every anonymous game in $L(n,m,\delta)$ for $\delta = \ep/2m$ admits a pure $\ep$--equilibrium.
\end{theorem}
\begin{corollary}
Fix $m$ and let $\delta:\bbn\rightarrow [0,1]$ be such that $\delta(n)=o(1)$. Then for every $\ep>0$, there is $N$ such that every game with $n\ge N$ players, $m$ strategies for each player and Lipschitz constant smaller than $\delta(n)$ admits a pure $\ep$-equilibrium.\end{corollary}

Before the proof we note that similar results to that of Theorem \ref{thm-anonymous} appeared in previous papers: Rashid \cite{Ras-83} was the first to prove an asymptotic (in the number of players) purification result for anonymous games; our proof follows the footsteps of his, the main difference being that we start with finite sets of strategies as a primitive while in Rashid's formulation the strategy set of each player is the unit simplex. See also \cite{Car-04}, \cite[Lemma 5]{Car-Pod-09} and \cite{Car-Woo-09} for related results.

Daskalakis and Papadimitriou \cite[Theorem 2.1]{Das-Pap-07} proved that
every anonymous game in $L(n,m,\delta)$ for $\delta=\ep/O(m^2)$ admits pure $\ep$--equilibrium, and conjectured that the $O(m^2)$ can be improved to $O(m)$; Theorem \ref{thm-anonymous} confirms their conjecture. The difference between our proof and theirs can be roughly described as follows: Our approach is to build an auxiliary game in which the strategy set of each player is extended to the entire unit simplex, use Nash's Theorem to obtain an equilibrium of the auxiliary game, and then use the Shapley-Folkman theorem to ``purify'' the equilibrium without changing the utilities too much. Daskalakis and Papadimitriou, on the other hand, extend the best-response function to a continuous function on the simplex, use Brower's theorem to obtain a fixed point of this function, and then use Hall's marriage lemma for the ``purification'' procedure.
\begin{proof}[Proof of Theorem~\ref{thm-anonymous}]
We denote by $S$ the common set of strategies, so that $|S|=m$ and $A=S^n$. We consider the vector space $\bbr^S$ equipped with the standard basis $\{\bfone^s|s\in S\}$ and the $L_1$ norm $\| \cdot\|_1$. Let $\Delta_{n}(S)$ be the set of all elements $\{x[s]\}_{s\in S} \in\bbr^S$ such that $x[s]\geq 0$ for every $s\in S$ and $\sum_s x[s]=n$. Thus, $\Delta_n(S)$ is an $m-1$ dimensional simplex in $\bbr^S$, whose vertices are $n\bfone^s$. Elements of $\Delta_n(S)$ with integral entries are called \emph{distributions}. The set of all distributions in $\Delta_n(S)$ is denoted $D_n(S)$. One can think about $D_n(S)$ as the set of possible ways to distribute $n$ identical balls in $m$ cells.


Fix a strategy profile $a=(s_1,\dots,s_n)\in A$ and a player $i$. Let $d(a_{-i})=\sum_{j\neq i}\bfone^{s_j} \in D_{n-1}(S)$ be the \emph{distribution induced by $a_{-i}$}. Since the game is anonymous, if $d(a_{-i})=d(a'_{-i})$ then player $i$ gets the same payoff against $a_{-i}$ and $a'_{-i}$. Thus, we can describe the game by payoff functions $\{F_i:S\times D_{n-1}(S) \rightarrow \mathbb{R}\}_{i=1}^n$, where $F_i(s,d)$ is the payoff for player $i$ under a strategy profile in which player $i$ plays $s$ and the distribution of her opponent's strategies is $d$.


If the Lipschitz constant of an anonymous game is $\delta$ then the payoff functions satisfy the following Lipschitz property:
\begin{equation}\label{liphshitz-capital-f}
|F_i(s,d)-F_i(s,d')|\leq \frac{\delta \|d-d'\|_1}{2}
\end{equation}
for every $s\in S$ and $d,d'\in D_{n-1}(S)$. We extend the domain of each $F_i(s,\cdot)$ from $D_{n-1}(S)$ to $\Delta_{n-1}(S)$, so that the Lipschitz property is preserved, for example by defining
\[F_i(s, x)=\max_{d\in D_{n-1}(S)} \left\{F_i(s,d)-\delta\|d-x\|_1/2 \right\} \]
for every $x\in \Delta_{n-1}(S)$.

Consider now the auxiliary $n$-player game in normal form where the strategy set of each player is $\Delta_1(S)$ and the payoff for player $i$ under strategy profile $p=(p_1,\dots,p_n)$ is given by
\[R_i(p)=\sum_{s\in S}p_i[s] F_i\left(s,\sum_{j\neq i}p_j\right).\]
The strategy set of every player is compact and convex, the payoff functions are jointly continuous, and each function is linear in a player's own strategy. Therefore the game admits a Nash equilibrium $\bar p=(\bar p_1,\dots,\bar p_n)$. Because of the linearity in own strategy, the equilibrium property implies that
\begin{equation}\label{equilibrium} \support (\bar p_i)\subseteq\argmax_{s\in S} F_i\left(s,\sum_{j\neq i}\bar p_j\right)\end{equation}
for every player $i$.

By the Shapley-Folkman Theorem (e.g., \cite{Zho-93}), we can write $\sum_i \bar p_i = \sum_{i\in E} q_i + \sum_{i\in E^c} \bfone^{\bar s_i}$, where $E$ is a set of $m-1$ players, $q_i\in\Delta_1(S)$ for each $i\in E$ and $\bar s_i\in \support(\bar p_i)$ for each $i\in E^c$. By replacing each $q_i$, $i\in E$, by some $\bar s_i\in \support (\bar p_i)$ we get a pure strategy profile $\bar a=(\bar s_1,\dots,\bar s_n)$ such that $\left\|\sum_i \bfone^{\bar s_i}-\sum_i\bar p_i \right\|_1 \leq 2(m-1)$. In particular, it follows that
\begin{equation}\label{sf-opponents}
\left\|\sum_{j\neq i} \bfone^{\bar s_j}-\sum_{j\neq i}\bar p_j \right\|_1\leq 2m\end{equation}
for every player $i$.

We claim that the strategy profile $\bar a$ is a $2m\delta$-equilibrium. Indeed, for every player $i$ and every strategy $s\in S$,
\begin{eqnarray*}
f_i(s,\bar a_{-i}) & = & F_i \left(s,\sum_{j\neq i}\bfone^{\bar s_j} \right) \leq F_i \left(s,\sum_{j\neq i}\bar p_j \right) + m\delta \leq  F_i \left(\bar s_i,\sum_{j\neq i}\bar p_j \right) + m\delta \leq \\ & &   F_i \left(\bar s_i,\sum_{j\neq i}\bfone^{\bar s_j} \right) + 2m\delta  = f_i(\bar s_i,\bar a_{-i})+2m\delta.
\end{eqnarray*}
The equalities follow from the definition of $F_i$. The first inequality follows from~\eqref{liphshitz-capital-f} and~\eqref{sf-opponents}, the second from the fact that $\bar s_i\in\support (\bar p_i)$  and~\eqref{equilibrium}, and the third again from~\eqref{liphshitz-capital-f} and~\eqref{sf-opponents}.
\end{proof}

As a last remark in this section we note that the self-purification argument in the proof of Theorem \ref{yes} fails when the Lipschitz constant of the game is not sufficiently small, \emph{even if the game is anonymous}. This is illustrated in the following example.
\begin{example}
Consider a town with $2n+1$ individuals (players) and one restaurant. Each player should decide between going out for dinner in the restaurant (strategy $R$) or cook her own dinner at home ($H$). Players like the restaurant to be populated but not crowded. Specifically, the payoff to a player that chooses $R$ when $k$ out of her $2n$ opponents also choose $R$ is given by
$$ g(k)=\begin{cases}1, &\text{ if }|k-n|\leq 0.477\sqrt{n}\\
\Bigl(1 - \delta\bigl(|k-n|-0.477\sqrt{n}\bigr)\Bigr)^+,&\text{ otherwise,}\end{cases}$$
where $\delta > 0$ is a small number. A player that chooses to stay home ($H$) gets a fixed payoff of $E(g(X))$ (independently of the choices of other players), where $X$ is a random variable with binomial distribution $X\sim Bin (2n,1/2)$.
\end{example}
First, notice that this is an anonymous game (actually this game is even symmetric in the sense that all the players have the same payoff function). Second, the Lipschitz constant of this game is $\delta$. Third, the strategy profile $\mu$ in which all players mix between $R$ and $H$ with equal probabilities is an equilibrium.

The function $g$ was chosen such that, when $n$ is large, $E(g(X))$ is approximately 1/2; this follows from the central limit theorem. Thus, each player can guarantee a payoff close to 1/2 by staying at home. It follows that, for $\ep=1/4$, a strategy profile is an $\ep$-equilibrium if and only if the number of players $k$ that choose $R$ is either zero or in the interval where $g(k)$ is close to 1/2. But when the equilibrium $\mu$ is played the probability of this event  diminishes when $n$ increases. Thus, the random realized profile will typically not be a $1/4$-equilibrium.

\appendix
\section{From pure $\ep$-equilibrium to exact mixed equilibrium}\label{sec-converse}

Theorem \ref{yes} states that if the Lipschitz constant of a game is small enough then pure $\ep$-equilibrium exists. This result is an assertion about relationships between linear inequalities: It says that if a certain finite number of linear inequalities about the entries of a payoff matrix (inequalities that assert that the Lipschitz constant is smaller than $\delta$) are satisfied, then there is a strategy profile for which another finite number of linear inequalities (inequalities that assert that this profile is an $\ep$-equilibrium) is satisfied. More formally, Theorem \ref{yes} is a first order sentence in the language of real numbers with addition (without multiplication)\cite{TLTC-10}.

However, the proof of Theorem \ref{yes} is based on Nash's theorem of existence of mixed equilibrium, and therefore on Brouwer's fixed point theorem. It may seem dubious that an appeal to such powerful theorems is required (note that the much more trivial Proposition \ref{sure} clearly does not rely on a fixed point argument). To show the non-triviality of the bound on the Lipschitz constant in Theorem~\ref{yes}, we show below that in fact Theorem~\ref{yes} also implies Nash's theorem via an elementary argument. A similar argument appears in Schmeidler \cite[pp. 298-299]{Sch-73} for the non-atomic setup.
\begin{claim*}Assume Theorem~\ref{yes}. Then every finite normal-form game admits a mixed strategy Nash equilibrium.\end{claim*}
\begin{proof}[Sketch of Proof]
Fix a normal form game $G$ with $n$ players and strategy sets $A_1,\dots,A_n$. Let $m$ be such that $|A_i|\leq m$ for every $i$, and fix $\ep>0$. Let $L$ be a sufficiently large integer and consider the game $G'$ with $n\cdot L$ players divided into groups $(T_1,\dots,T_n)$ of size $L$ each, where players in $T_i$ have strategy sets $A_i$. In the game $G'$ every $n$-tuple of players $(t_1,\dots,t_n)$ where $t_i\in T_i$ play the game $G$, and each player must use the same strategy in all the games in which he participates. The payoff to a player is the average of all the payoffs he receives. If $\delta$ is the Lipschitz constant of the original game $G$, then the Lipschitz constant of $G'$ is at most $\delta/L$. Thus, the game $G'$ has $nL$ players and for sufficiently large $L$ its Lipschitz constant is smaller than $\ep/\sqrt{8nL\log (2mnL)}$, and therefore by Theorem~\ref{yes} $G'$ admits a pure $\ep$-equilibrium. If $\mu_i\in\Delta(A_i)$ is the distribution of strategies played by players in $T_i$ according to the pure $\ep$-equilibrium profile of $G'$, then $(\mu^1,\dots,\mu^n)$ is a mixed $\ep$-equilibrium in $G$. Thus, we proved that $G$ admits a mixed $\ep$-equilibrium for every $\ep>0$. An accumulation point of these $\ep$-equilibria is a mixed Nash equilibrium of $G$.\end{proof}
\bibliographystyle{plain}
\bibliography{large}

\begin{thebibliography}{10}

\bibitem{Aln-Smo-00}
Nabil~I. Al-Najjar and Rann Smorodinsky.
\newblock Pivotal players and the characterization of influence.
\newblock {\em J. Econom. Theory}, 92(2):318--342, 2000.

\bibitem{Alo-Spe-08}
Noga Alon and Joel~H. Spencer.
\newblock {\em The probabilistic method}.
\newblock Wiley-Interscience Series in Discrete Mathematics and Optimization.
  John Wiley \& Sons Inc., Hoboken, NJ, third edition, 2008.
\newblock With an appendix on the life and work of Paul Erd{\H{o}}s.

\bibitem{Car-04}
Guilherme Carmona.
\newblock On the purification of {N}ash equilibria of large games.
\newblock {\em Econom. Lett.}, 85(2):215--219, 2004.

\bibitem{Car-08}
Guilherme Carmona.
\newblock Purification of {B}ayesian-{N}ash equilibria in large games with
  compact type and action spaces.
\newblock {\em J. Math. Econom.}, 44(12):1302--1311, 2008.

\bibitem{Car-Pod-09}
Guilherme Carmona and Konrad Podczeck.
\newblock On the existence of pure-strategy equilibria in large games.
\newblock {\em Journal of Economic Theory}, 144:1300--1319, 2009.

\bibitem{Car-Pod-10-1}
Guilherme Carmona and Konrad Podczeck.
\newblock Ex--post stability of {B}ayes--{N}ash equilibria of large games.
\newblock {\em Games and Economic Behavior}, 74(1):418--430, 2012.

\bibitem{Car-Woo-09}
Edward Cartwright and Myrna Wooders.
\newblock On equilibrium in pure strategies in games with many players.
\newblock {\em Internat. J. Game Theory}, 38(1):137--153, 2009.

\bibitem{Das-Pap-07}
Constantinos Daskalakis and Christos Papadimitriou.
\newblock Computing equilibria in anonymous games.
\newblock {\em 48th Annual IEEE Symposium on Foundations of Computer Science},
  pages 83--93, 2007.

\bibitem{Deb-Kal-10}
Joyee Deb and Ehud Kalai.
\newblock Stability in large bayesian games with heterogeneous players.
\newblock {\em Working paper}, 2010.

\bibitem{Gra-Rei-10}
Ronen Gradwohl and Omer Reingold.
\newblock Partial exposure in large games.
\newblock {\em Games and Economic Behavior}, 68(2):602--613, 2010.

\bibitem{Jac-08}
Matthew~O. Jackson.
\newblock {\em Social and economic Networks}.
\newblock Princeton University Press, 2008.

\bibitem{Kal-04}
Ehud Kalai.
\newblock Large robust games.
\newblock {\em Econometrica}, 72(6):1631--1665, 2004.

\bibitem{Led-01}
Michel Ledoux.
\newblock {\em The concentration of measure phenomenon}, volume~89 of {\em
  Mathematical Surveys and Monographs}.
\newblock American Mathematical Society, Providence, RI, 2001.

\bibitem{Mil-Rob-90}
Paul Milgrom and John Roberts.
\newblock Rationalizability, learning, and equilibrium in games with strategic
  complementarities.
\newblock {\em Econometrica}, 58(6):1255--1277, 1990.

\bibitem{Mon-Sha-96}
Dov Monderer and Lloyd~S. Shapley.
\newblock Potential games.
\newblock {\em Games Econom. Behav.}, 14:124--143, 1996.

\bibitem{TLTC-10}
The~Leisure of~the Theory~Class.
\newblock Brouwer, tarski and the existence of nash equilibrium.
\newblock
  \href{http://theoryclass.wordpress.com/2010/10/05/brouwer-tarski-and-the-existence-of-nash-equilibrium/}{http://theoryclass.wordpress.com/2010/10/05/brouwer-tarski-and-the-existence-of-nash-equilibrium/},
  2010.

\bibitem{Ras-83}
Salim Rashid.
\newblock Equilibrium points of nonatomic games: {A}symptotic results.
\newblock {\em Econom. Lett.}, 12(1):7--10, 1983.

\bibitem{Ros-73}
Robert~W. Rosenthal.
\newblock A class of games possessing pure-strategy {N}ash equilibria.
\newblock {\em Internat. J. Game Theory}, 2:65--67, 1973.

\bibitem{Rub-91}
Ariel Rubinstein.
\newblock Comments on the interpretation of game theory.
\newblock {\em Econometrica}, 59(4):909--924, 1991.

\bibitem{Sch-73}
David Schmeidler.
\newblock Equilibrium points of nonatomic games.
\newblock {\em J. Statist. Phys.}, 7:295--300, 1973.

\bibitem{Zho-93}
Lin Zhou.
\newblock A simple proof of the shapley-folkman theorem.
\newblock {\em Econ. Theorey}, 3:371--372, 1993.

\end{thebibliography}

\end{document}